\documentclass[12pt]{amsart}

\usepackage[text={425pt,650pt},centering]{geometry}
\usepackage{color}
\usepackage{esint,amssymb}
\usepackage{graphicx}
\usepackage{tikz}

\newtheorem{theorem}{Theorem}
\newtheorem{proposition}[theorem]{Proposition}
\newtheorem{lemma}[theorem]{Lemma}
\newtheorem{corollary}[theorem]{Corollary}

\theoremstyle{definition}

\newtheorem{conjecture}[theorem]{Conjecture}

\newcommand{\cref}[1]{Corollary~\ref{c.#1}}

\numberwithin{equation}{section}
\numberwithin{theorem}{section}

\newcommand{\R}{\mathbb{R}}

\renewcommand{\bar}{\overline}
\renewcommand{\tilde}{\widetilde}

\begin{document}

\title[Stable solutions with the half Laplacian]{On stable solutions 
for boundary reactions:\\ a De Giorgi-type result in dimension 4+1}

\begin{abstract}
We prove that every bounded stable solution  of 
\[ (-\Delta)^{1/2} u  + f(u) =0  \qquad \mbox{in }\R^3\]
is a 1D profile, i.e., $u(x)= \phi(e\cdot x)$ for some $e\in \mathbb S^2$, where $\phi:\R\to \R$ is a nondecreasing  bounded stable solution in dimension one. 
This proves the De Giorgi conjecture in dimension $4$ for the half-Laplacian.
Equivalently, 
we give a positive answer to the De Giorgi conjecture for boundary reactions in $\R^{d+1}_+=\R^{d+1}\cap \{x_{d+1}\geq 0\}$ when $d = 4$, by proving that all critical points of 
$$
\int_{\{x_{d+1\geq 0}\}} \frac12 |\nabla U|^2 \,dx\, dx_{d+1} + \int_{\{x_{d+1}=0\}}  \frac 1 4 (1-U^2)^2  \,dx
$$
that are monotone in $\R^d$ (that is, up to a rotation, $\partial_{x_d} U>0$) are one dimensional.

Our result is analogue to the fact that stable embedded minimal surfaces in $\R^3$ are planes. Note that the corresponding result about stable solutions to the classical Allen-Cahn equation (namely, when the half-Laplacian
is replaced by the classical Laplacian)
is still open.

\end{abstract}

\author[A. Figalli]{Alessio Figalli}
\address{ETH Z\"urich, Department of Mathematics, R\"amistrasse 101, 8092 Z\"urich, Switzerland.}
\email{alessio.figalli@math.ethz.ch}

\author[J. Serra]{Joaquim Serra}
\address{ETH Z\"urich, Department of Mathematics, R\"amistrasse 101, 8092 Z\"urich, Switzerland.}
\email{joaquim.serra@math.ethz.ch}

%\subjclass[2010]{35B27, 35J60, 60F17}
\date{\today}

\maketitle

\section{introduction}

\subsection{De Giorgi conjecture on the Allen-Cahn equation}
In 1978, De Giorgi stated the following famous conjecture \cite{DG}:
\begin{conjecture}\label{DGconj}{\it 
Let $u\in C^2(\R^{d})$ be a solution of the Allen-Cahn equation
\begin{equation}\label{classic}
-\Delta u = u-u^3, \quad  |u|\le1,
\end{equation} 
satisfying $\partial_{x_{d}} u >0$.  Then, if $d\le 8$,  all level sets $\{u=\lambda\}$ of $u$ must be  hyperplanes.}
\end{conjecture}
To motivate this conjecture, we need to explain its relation to minimal surfaces.

\subsection{Allen-Cahn vs. minimal surfaces}
\label{sect:AC-min}
It is well-known that  \eqref{classic} is the condition of vanishing first variation for the Ginzburg-Landau energy 
\[
\mathcal E_1(v)  := \int_{\R^{d}} \bigg(\,  \frac 1 2  |\nabla v|^2  + \frac 1 4 (1-v^2)^2 \bigg) dx.
\]
By scaling, if $u$ is a local minimizer of $\mathcal E_1$, then $u_\varepsilon(x) := u (\varepsilon^{-1}x)$  are local minimizers of the $\varepsilon$-energy 
\begin{equation}
\label{eq:rescaled E}
\mathcal E_{1,\varepsilon}(v) := \frac{1}{\varepsilon}\int_{\R^{d}} \bigg(\,  \frac{\varepsilon^2}2  |\nabla v|^2  + \frac 1 4 (1-v^2)^2 \bigg) dx.
\end{equation}
In \cite{Modica, MM}, Modica and Mortola established the $\Gamma$-convergence of $\mathcal E_{1,\varepsilon}$ to the perimeter functional as $\varepsilon \downarrow 0$. As a consequence, the rescalings $u_\varepsilon$ have a subsequence $u_{\varepsilon_k}$ such that
\[u_{\varepsilon_k} \rightarrow \chi_E -\chi_{E^c}\quad \mbox{in  }L^1_{\rm loc},\]
and $E$ is a local minimizer of the perimeter in $\R^{d}$.
This result was later improved by Caffarelli and Cordoba \cite{CCor},
who showed a density estimate for minimizers of $\mathcal E_{1,\varepsilon}$,
and proved that the super-level sets $\{u_{\varepsilon_k}\ge \lambda\}$ converge locally uniformly (in the sense of Hausdorff distance) to $E$ for each fixed $\lambda\in(-1,1)$. 
Hence, at least heuristically, minimizers of  $\mathcal E_{1,\varepsilon}$ for $\varepsilon$ small should behave similarly to sets of minimal perimeter.

\subsection{Classifications of entire minimal surfaces and De Giorgi conjecture}
\label{sect:min surf DG conj}
Here we recall some well-known facts on minimal surfaces:\footnote{Note that, here and in the sequel, the terminology ``minimal surface'' denotes a critical point of the area functional (in other words, a surface with zero mean curvature).}
\begin{enumerate}
\item[(i)] If $E$ is a local minimizer of the perimeter in $\R^{d}$ with $d\le 7$, then $E$ is a halfspace.
\item[(ii)]   The Simons cone $\bigl\{x_1^2+x_2^2+x_3^2+x_4^2<x_5^2+x_6^2+x_7^2+x_8^2\bigr\}$ is a local minimizer of perimeter in $\R^8$ which is not a halfspace.
\end{enumerate}
Also, we recall that these results hold in one dimension higher if we restrict to minimal graphs:
\begin{enumerate}
\item[(i')] If $E=\bigl\{x_{d}\ge h(x_1,\dots, x_{d-1}) \bigr\}$ is a epigraph in $\R^{d}$,  $\partial E$ is a minimal surface, and $d\le 8$ , then $h$ is affine (equivalently, $E$ is a halfspace).
\item[(ii')]  There is a non-affine entire minimal graph in dimension $d=9$.
\end{enumerate}
These assertions combine several classical results. The main contributions leading to (i)-(ii)-(i')-(ii') are the landmark papers of De Giorgi \cite{DG1,DG-Bern} (improvement of flateness -- Bernstein theorem for minimal graphs), Simons \cite{Simons} (classification of stable minimal cones), and Bombieri, De Giorgi, and Giusti \cite{BDG} (existence of a nontrivial minimal graph in dimension $d=9$, and minimizing property of the Simons cone).

Note that, in the assumptions 
of
Conjecture \ref{DGconj}, the function $u$ satisfies $\partial_{x_{d}}u>0$,
a condition that implies that the super-level sets $\{u\ge\lambda\}$ are epigraphs.
Thus, if we assume that $d\leq 8$, it follows by (ii') and the discussion in Section \ref{sect:AC-min} 
that the level sets of $u_\varepsilon(x)=u(\varepsilon^{-1}x)$ should be close to a hyperplane for $\varepsilon\ll 1$.
Since
$$
\{u_\varepsilon=\lambda\}=\varepsilon\{u=\lambda\},
$$
this means that all blow-downs of $\{u=\lambda\}$ (i.e., all possible limit points of $\varepsilon\{u=\lambda\}$ as $\varepsilon\downarrow0$) are hyperplanes. Hence, the conjecture of De Giorgi asserts that, for this to be true, the level sets of $u$ had to be already hyperplanes.

\subsection{Results on the De Giorgi conjecture}
Conjecture \ref{DGconj} was first proved, about twenty years after it was raised, in dimensions  $d=2$ and $d=3$, by Ghoussoub and Gui \cite{GG} and  Ambrosio and Cabr\'e \cite{AmbrC}, respectively.
Almost ten years later, in the celebrated paper  \cite{Savin}, Savin attacked the conjecture in the dimensions $4\le d\le 8$, and he succeeded in proving it under the additional assumption
\begin{equation}\label{limits}
 \lim_{x_{d}\to \pm \infty} u = \pm 1.
\end{equation}
Short after, Del Pino, Kowalczyk, and Wei \cite{dPKW} established the existence of a counterexample in dimensions  $d\ge 9$ .

It is worth mentioning that the extra assumption \eqref{limits} in \cite{Savin} is only used to guarantee that $u$ is a local minimizer of $\mathcal E_1$. Indeed, while in the case of minimal surfaces epigraphs are automatically minimizers of the perimeter, the same holds for monotone solutions of \eqref{classic} only under the additional assumption \eqref{limits}.

\subsection{Monotone vs. stable solutions}
\label{sect:monotone stable}
Before introducing the problem investigated in this paper, we make a connection between monotone and stable solutions.

 It is well-known (see \cite[Corollary 4.3]{AAC}) that monotone solutions to \eqref{classic} in $\R^{d}$ are stable solutions, i.e., the second variation of $\mathcal E_1$ is nonnegative.
Actually, in the context of monotone solutions it is natural to consider the two limits $$u^{\pm} := \lim_{x_{d} \to \pm\infty} u,$$ which are functions of the first $d-1$ variables $x_1, \dots, x_{d-1}$ only, and one can easily prove that $u^\pm$  are stable solutions of \eqref{classic} in $\R^{d-1}$. If one could show that these functions
are 1D, then the  results of Savin \cite{Savin}
would imply that $u$ was also 1D.

In other words, the following implication holds:
$$
\text{stable solutions to \eqref{classic} in $\R^{d-1}$ are 1D}
\quad \Rightarrow \quad  \text{De Giorgi conjecture holds in $\R^{d}$.}
$$

\subsection{Boundary reaction and line tension effects}
A natural variant of the Ginzburg-Landau energy, first introduced by Alberti, Bouchitt\'e, and Seppecher in \cite{ABS} and then studied by  
Cabr\'e and Sol\`a-Morales in \cite{C-SM},
consists in studying a Dirichlet energy with boundary potential on a half space $\R^{d+1}_+:=\{x_{d+1}> 0\}$ (the choice of considering $d+1$ dimensions will be clear by the discussion in the next sections). In other words, one considers the energy functional 
\[
\mathcal J(V) := \int_{\R^{d+1}_+}\frac12 |\nabla V|^2 \,dx\,dx_{d+1} + \int_{\{x_{d+1}=0\}} F(V) \,dx\,,
\]
where $F:\R\to \R$ is some potential. Then, the Euler-Lagrange equation corresponding to $\mathcal J$ is given by
\begin{equation}
\label{eq:EL d+1}
\left\{
\begin{array}{ll}
\Delta U=0 &\text{in }\R^{d+1}_+,\\
\partial_\nu U=-f(U) &\text{on }\{x_{d+1}=0\},
\end{array}
\right.
\end{equation}
where $f=F'$, and $\partial_\nu U=-\partial_{x_{d+1}}U$ is the exterior normal derivative. When
$f(U)= \sin(c\, U)$, $c \in \R$, the above problem is called the Peierls-Navarro equation and appears in a model of dislocation of crystals \cite{GM, Toland}. Also, the same equation is central for
the analysis of boundary vortices for soft thin films in \cite{K}.

\subsection{Non-local interactions}
To state the analogue of the De Giorgi conjecture in this context we first recall that, for a harmonic function $V$, the energy $\mathcal J$
can be rewritten in terms of its trace $v:=V|_{x_{d+1}=0}$. More precisely, 
a classical computation shows that (up to a multiplicative dimensional constant)
the Dirichlet energy of $V$ is equal to the $H^{1/2}$ energy of $v$:
$$
\int_{\R^{d+1}_+} |\nabla V|^2 \,dx_1\cdots dx_d\, dx_{d+1}
=\iint_{\R^{d}\times \R^d}  \frac{|v(x)-v(y)|^2}{|x-y|^{d+1}}\,dx\,dy
$$
(see for instance \cite{CafSi}). 
Hence, instead of $\mathcal J$, one can consider the energy functional
$$
\mathcal E(v):=\iint_{\R^{d}\times \R^d} 
\frac12 \frac{|v(x)-v(y)|^2}{|x-y|^{d+1}}\,dx\,dy
+\int_{\R^d}F(v(x))\,dx
$$
and because harmonic functions minimize the Dirichlet energy, one can easily prove  that
$$
\text{$U$ is a local min. of $\mathcal J$ in $\R^{d+1}_+$}\quad \Leftrightarrow\quad
\left\{
\begin{array}{l}
\text{$u= U|_{x_{d+1}=0}$ is a local min. of $\mathcal E$ in $\R^d$}\\
\text{and $U$ is the harmonic extension of $u$.}
\end{array}
\right.
$$
Hence, in terms of the function $u$,
the Euler-Lagrange equation \eqref{eq:EL d+1} corresponds to the first variation of $\mathcal E$, namely
\begin{equation}\label{eq}
(-\Delta)^{1/2} u  + f(u) =0\quad \mbox{in }\R^d,
\end{equation} 
where
$$
(-\Delta)^{1/2}u(x)=\text{p.v.}\int_{\R^d}\frac{u(x)-u(y)}{|x-y|^{d+1}}\,dy.
$$

\subsection{$\Gamma$-convergence of nonlocal energies to the classical perimeter, and the De Giorgi conjecture for the $1/2$-Laplacian}
Analogously to what happens with the classical Allen-Cahn equation, there is a  connection between solutions of $(-\Delta)^{1/2} u = u-u^3$  and  minimal surfaces. Namely, if $u$ is a local minimizer of $\mathcal E$ in $\R^d$ with $F(u)= \frac 1 4 (1-u^2)^2$, then the rescaled functions $u_\varepsilon(x) = u (\varepsilon^{-1}x)$  are local minimizers of the $\varepsilon$-energy 
\[
\mathcal E_{\varepsilon}(v)  :=\frac{1}{\varepsilon \log(1/\varepsilon)} \bigg(  \iint_{\R^d\times \R^d} \frac{\varepsilon}{2} \frac{|v(x)-v(y)|^2}{|x-y|^{d+1}} \,dx\,dy   +  \int_{\R^d}\frac 1 4 (1-v^2)^2\, dx\bigg).
\]
As happened for the energies $\mathcal E_{1,\varepsilon}$ in \eqref{eq:rescaled E}, the papers \cite{ABS,  MdM}  established the $\Gamma$-convergence of $\mathcal E_{\varepsilon}$ to the perimeter functional as $\varepsilon \downarrow 0$, as well as
 the existence of a subsequence $u_{\varepsilon_k}$ such that
\[u_{\varepsilon_k} \rightarrow \chi_E -\chi_{E^c}\quad \mbox{in  }L^1_{\rm loc},\]
where $E$ is a local minimizer of the  perimeter in $\R^d$.
Moreover, Savin and Valdinoci \cite{SV-1} proved density estimates for minimizers of $\mathcal E_{\varepsilon}$, implying that $\{u_{\varepsilon_k}\ge \lambda\}$ converge locally uniformly  to $E$ for each fixed $\lambda\in(-1,1)$.

Hence, the discussion in Section \ref{sect:min surf DG conj} motivates the validity of the De Giorgi conjecture when $-\Delta$ is replaced with $(-\Delta)^{1/2}$,
namely: 
\begin{conjecture}\label{DGconj2}{\it 
Let $u\in C^2(\R^{d})$ be a solution of the fractional Allen-Cahn equation
\begin{equation}\label{fract}
(-\Delta)^{1/2} u = u-u^3, \quad  |u|\le1,
\end{equation} 
satisfying $\partial_{x_{d}} u >0$. Then, if $d\le 8$, all level sets $\{u=\lambda\}$ of $u$ must be  hyperplanes.}
\end{conjecture}
In this direction, Cabr\'e and Sol\`a-Morales proved the conjecture for $d=2$ \cite{C-SM}. Later, Cabr\'e and Cinti \cite{CC1} established Conjecture \ref{DGconj2} for $d=3$. 
Very recently, under the additional assumption \eqref{limits}, Savin has announced in \cite{S-new}
a proof of Conjecture \ref{DGconj2} in the remaining dimensions $4\le d\le 8$.
Thanks to the latter result, the relation between monotone and stable solutions explained in Section \ref{sect:monotone stable} holds also in this setting.
%, they established Conjecture \ref{DGconj2} for $d=3$.

\subsection{Stable solutions vs. stable minimal surfaces}
Exactly as in the setting of Conjecture \ref{DGconj}, given $u$ as in Conjecture \ref{DGconj2} it is natural to introduce the two limit functions $u^{\pm} := \lim_{x_{d} \to \pm\infty} u$. These functions depend only on the first $d-1$ variables $x_1, \dots, x_{d-1}$, and are stable solutions of \eqref{fract} in $\R^{d-1}$.

As mentioned at the end of last section,
 the classification of stable solutions to \eqref{fract} in $\R^{d-1}$, $3\le d-1\le 7$,
 together with the improvement of flatness for $(-\Delta)^{1/2} u = u-u^3$ announced in \cite{S-new},
  would imply the full Conjecture \ref{DGconj2} in $\R^{d}$.
%as an application of Theorem \ref{thm} and the improvement of flatness for $(-\Delta)^{1/2} u = u-u^3$ announced in \cite{S-new}, we obtain the following:

%As for \eqref{classic}, the classification of stable solutions to \eqref{fract} in $\R^{d-1}$, $3\le d-1\le 7$,  would imply the full Conjecture \ref{DGconj} in $\R^{d}$.

The difficult problem of classifying stable solutions of \eqref{fract} (or of \eqref{classic}) is connected to the following well-known conjecture for minimal surfaces:
\begin{conjecture}\label{conjstableminimal}{\it 
Stable embedded minimal hypersurfaces in  $\R^d$ are hyperplanes as long as $d\le 7$.}
\end{conjecture}
A positive answer to this conjecture is only known to be true in dimension $d=3$,  a result of Fischer-Colbrie and Schoen \cite{FishS} and Do Carmo and Peng \cite{dCP}.
 %Unfortunately, the two different proofs in \cite{FishS} \cite{dCP} ara very specific to the case of two dimensional surfaces.
%Both proofs from rely on the area bound for the area of geodesic balls due to Pogorelov \cite{P},
%and it is proved by choosing a suitable conical test function in the stability estimate, together with the Gauss-Bonnet formula. Hence, the argument is very specific to the case of two dimensional surfaces.
Note that, for minimal cones, the conjecture is true (and the dimension 7 sharp) by the results of Simons \cite{Simons} and Bombieri, De Giorgi, and Giusti \cite{BDG}. 

Conjecture \ref{conjstableminimal} above suggests a ``stable De Giorgi conjecture'': 
\begin{conjecture}\label{DGconj3}{\it 
Let $u\in C^2(\R^{d})$ be a stable solution of \eqref{classic} or of \eqref{fract}.
Then, if $d\le 7$, all level sets $\{u=\lambda\}$ of $u$ must be  hyperplanes.}
\end{conjecture}
As explained before, the validity of this conjecture would imply both Conjectures \ref{DGconj} and \ref{DGconj2}.

\subsection{Results of the paper}
As of now, Conjecture \ref{DGconj3} has been proved only for $d=2$ (see \cite{BCN, GG} for \eqref{classic}, and
\cite{C-SM} for \eqref{fract}).
The main result of this paper establishes its validity for \eqref{fract} and $d=3$, a case that heuristically corresponds to the classification in $\R^3$ of stable minimal surfaces of \cite{FishS}.
Note that, for the classical case \eqref{classic}, Conjecture \ref{DGconj3} in the case $d=3$ is still open.

This is our main result:
\begin{theorem}\label{thm}
Let $u$ be a stable solution of \eqref{eq} with $d=3$ such that $|u|\leq 1$, and assume that $f\in C^{0,\alpha}([-1,1])$ for some $\alpha>0$. 
Then $u$ is  1D profile, namely,  $u(x)= \phi(e\cdot x)$ for some $e\in \mathbb S^2$, where $\phi:\R\to \R$ is a nondecreasing bounded stable solution to \eqref{eq} in dimension one. 
\end{theorem}
As explained before,
as an application of Theorem \ref{thm} and the improvement of flatness for $(-\Delta)^{1/2} u = u-u^3$ announced in \cite{S-new}, we obtain the following:
\begin{corollary}\label{corDG}
Conjecture \ref{DGconj2} holds true  in dimension $d=4$.
\end{corollary}

A key ingredient behind the proof of Theorem \ref{thm} is the following general energy estimate which holds in every dimension $d\ge 2$:
\begin{proposition}\label{prop}
Let $R\ge1$, $M_o\ge 2$, and $\alpha\in (0,1)$. Assume that $u$ be a stable solution of 
\begin{equation}\label{star}
(-\Delta)^{1/2} u +f(u) =0,\quad  |u|\le 1 \qquad \mbox{in }B_{R}\subset \R^d,
\end{equation} 
where $f:[-1,1]\to \R$ satisfies $\|f\|_{C^{0,\alpha}([-1,1])} \le M_o$. Then there exists a constant $C>0$,
depending only on $d$ and $\alpha$, such that
\[ 
\int_{B_{R/2}} |\nabla u| \le C R^{d-1} \log(M_o R)
\]
and
\[
\iint_{\R^d\times \R^d\setminus B_{R/2}^c\times B_{R/2}^c} \frac{|u(x)-u(y)|^2}{|x-y|^{d+1}} dxdy \le C R^{d-1} \log^2(M_o R).
\]
\end{proposition}
Note that the estimates in Proposition \ref{prop} differ from being sharp by a factor $\log(M_o R)$
(just think of the case when $u$ is a 1D profile).
However, for stable solutions of \eqref{star} in $\R^3$ we are able to bootstrap these non-sharp estimates to sharp ones, from which Theorem \ref{thm} follows easily.

\subsection{Structure of the paper}
In the next section we collect all the basic estimates needed for the proof of Proposition \ref{prop}.
Then, in Section \ref{sect:prop} we prove Proposition \ref{prop}. Finally, in Section \ref{sect:thm}
we prove Theorem \ref{thm}.
\\

{\it Acknowledgments:} both authors are supported by ERC Grant ``Regularity and Stability in Partial Differential Equations (RSPDE)''.

\section{Ingredients of the proofs}
We begin by introducing some notation.

Given $R>0$, we define the energy of a function inside $B_R\subset \R^d$ as
$$
\mathcal E(v;B_R) :=
\iint_{\R^d\times \R^d\setminus B_R^c\times B_R^c} \frac{1}{2}\frac{|v(x)-v(y)|^2}{|x-y|^{d+1}}\,dx\,dy
+\int_{B_R}F(v(x))\,dx,
$$
where $B_R^c=\R^d\setminus B_R$, and $F$ is a primitive of $f$.
Note that equation \eqref{star} is the condition of vanishing first variation for the energy functional $\mathcal E(\,\cdot\, ; B_R)$.

We say that a solution $u$ of \eqref{star} is \emph{stable} if the second variation at $u$ of $ \mathcal E$ is nonnegative, that is
\begin{equation}
\int_{B_R} \big( (-\Delta)^{1/2}\xi +f'(u)\xi\big)\xi \,dx \ge 0 \quad \mbox{for all }\xi\in C^2_c(B_R).
\end{equation}
Also, we say that $u$ is stable in $\R^d$ if it is stable in $B_R$ for all $R\ge 1$.

An important ingredient in our proof consists in considering variations of
a stable solution $u$ via
a suitable smooth $1$-parameter family of ``translation like'' deformations. This kind of idea has been first used by Savin and Valdinoci in \cite{SV, SV-mon}, and then in \cite{CSV,CCS}.
More precisely, given $R\geq 3$, consider the cut-off functions
\[
\varphi^{0}(x):= 
\begin{cases}
1 &\mbox{for } |x|\le \frac 1 2 \\
2-2|x| \quad &\mbox{for } \frac 1 2 \le  |x|\le 1\\
0 &\mbox{for } |x|\ge 1,
\end{cases}
\]
\[
\varphi^{1}(x):= 
\begin{cases}
1 &\mbox{for } |x|\le \sqrt R \\
2-2\frac{\log|x|}{\log R} \quad &\mbox{for } \sqrt R \le  |x|\le R\\
0 &\mbox{for }  |x| \ge R,
\end{cases}
\]
\[
\varphi^{2}(x):= 
\begin{cases}
1 &\mbox{for } |x|\le R_*\\
\displaystyle
1- \frac{\log\log |x|}{\log\log(R)}&\mbox{for }  R_*\le |x|\le R\\
0 &\mbox{for } |x|\geq R,
\end{cases}
\]
where $R^* := \exp(\sqrt{\log R})$.

For a fixed unit vector $\boldsymbol v\in \mathbb S^{n-1}$ define 
\begin{equation}\label{perturb}
\Psi_{t,\boldsymbol v}^i(z):= z + t\varphi^i  (z)\boldsymbol v,\qquad t \in \R,\quad  z \in \R^d, \quad i=0,1,2.
\end{equation}
Then, given a function $v:\R^d\to \R$ and $t\in(-1,1)$ with $|t|$ small enough (so that $\Psi_{t,\boldsymbol v}^i$ is invertible), we define the operator
\begin{equation}\label{perturb2}
\mathcal P^i_{t,\boldsymbol v} v(x) := v\bigl((\Psi_{t,\boldsymbol v}^i) ^{-1}(x)\bigr).
\end{equation}
Also, we use $\mathcal E^{\rm Sob}$ and $\mathcal E^{\rm Pot}$ to denote respectively the fractional Sobolev term and the Potential term appearing in the definition of $\mathcal E$:
$$
\mathcal E^{\rm Sob} (u; B_R) := \iint_{\R^d\times\R^d\setminus B_R^c \times B_R^c} \frac{|u(x)-u(y)|^2}{|x-y|^{d+1}}\,dx\,dy,
$$
$$
\mathcal E^{\rm Pot} (u; B_R) := \int_{B_R} F(u(x)) \,dx.
$$
We shall use the following bounds:
\begin{lemma}\label{quadratic}There exists a dimensional constant $C$ such that the following hold
for all $v:\R^d\to \R$, $|t|$ small, and $\boldsymbol v \in \mathbb S^{d-1}$:

(1) We have
\begin{equation}\label{1}
\mathcal E(\mathcal P^0_{t,\boldsymbol v} v;  B_1) + \mathcal E(\mathcal P^0_{-t,\boldsymbol v} v;B_1) - 2\mathcal E(v; B_1) \le C   t^2 \mathcal E^{\rm Sob}(v; B_2).
\end{equation}

(2) For  $R=2^{2k}$, $k \geq 1$, we have
\begin{equation}\label{2}
\mathcal E(\mathcal P^1_{t,\boldsymbol v}v;  B_R) + \mathcal E(\mathcal P^1_{-t,\boldsymbol v} v;B_R) - 2\mathcal E(v; B_R) \le C   \frac{t^2}{k^2} \sum_{j=1}^{k} \frac{\mathcal E^{\rm Sob}(v; B_{2^{k+j}})}{2^{2(k+j)}}.
\end{equation}

(3) For $R\ge 4$, 
\begin{equation}\label{3}
\mathcal E(\mathcal P^2_{t,\boldsymbol v} v;  B_R) + \mathcal E(\mathcal P^2_{-t,\boldsymbol v}v;B_R) - 2\mathcal E(v; B_R) \le C\frac{t^2}{\log\log R} \,\sup_{\rho\ge 2} \frac{\mathcal E^{\rm Sob}(v;B_\rho)} {\rho^2\log \rho}.
\end{equation}
\end{lemma}

\begin{proof}
The lemma follows as in \cite[Lemma 2.1]{CSV} and \cite[Lemma 2.3]{CCS}. However, since we do not have a precise reference for the estimates that we need, we give a sketch of  proof.
Note that, by approximation, it suffices to consider the case when $v\in C^2_c(\R^d)$.

First observe that, since $\boldsymbol v$ has unit norm, the Jacobian of the change of variables $z\mapsto \Psi_{t,\boldsymbol v}^i(z)$ is given 
by 
\begin{equation}\label{jac}
J_t^i(z):= |{\rm det}(D \Psi^i_{t,\boldsymbol v}(z))| = |{\rm det}({\rm Id} + t \nabla \varphi^i (z) \otimes \boldsymbol v)| = 1+t |\nabla \varphi^i(z) |. 
\end{equation}
Set
\begin{equation}
\label{eq:Ri}
R^i:=\left\{
\begin{array}{ll}
1&\text{if }i=0\\
R&\text{if }i=1,2.
\end{array}
\right.
\end{equation}
Then, performing the change of variables $x:=\Psi_{t,\boldsymbol v}^i(z)$, we get
\[
\mathcal E^{\rm Pot} (\mathcal P^i_{t,\boldsymbol v} v;  B_{R^i}) = \int_{B_{R^i}} 
 F\big(v\bigl((\Psi_{t,\boldsymbol v}^i) ^{-1}(x)\bigr)  \big)\,dx  =  \int_{B_{R^i}}   F(v(z)) \big(1+t |\nabla \varphi^i(z)| \big)\,dz,
\]
thus
\[
\mathcal E^{\rm Pot}(\mathcal P^i_{t,\boldsymbol v}v;  B_{R^i}) + \mathcal E^{\rm Pot}(\mathcal P^i_{-t,\boldsymbol v} v;B_{R^i}) - 2\mathcal E^{\rm Pot}(v; B_{R^i}) =0.
\]
Hence, we only need to estimate the second order incremental quotient of $\mathcal E^{\rm Sob}$.
To this aim, using the same change of variable and setting
\[A_{r}:=\R^n\times \R^n\setminus B_{r}^c\times B_{r}^c\qquad (r>0)\]
and $K(z) := |z|^{-(d+1)}$, we have (note that $\Psi_{t,\boldsymbol v}^i$ preserves $B_{R^i}$)
\begin{equation}\label{changeofvars}
\begin{split}
\mathcal E^{\rm Sob}(\mathcal P^i_{t,\boldsymbol v}v;  B_{R^i}) % &=\iint_{A_{R^i}} |u(\Psi_{R,\pm t}^{-1}(x))-u(\Psi_{R,\pm t}^{-1}(\bar x))|^2 K(x-\bar x)\,dx\,d\bar x\,\\
%&
= \iint_{A_{R^i}} |v(y)-v(\bar y)|^2 K\bigl( \Psi^{i}_{t, \boldsymbol v}(y)- \Psi^{i}_{t, \boldsymbol v}(\bar y)\bigr) \,J^i_{t}(y) \,dy \,J^i_{ t}(\bar y)\,d\bar y\,.
 \end{split}
\end{equation}
Recalling that
$\Psi^{i}_{t, \boldsymbol v}(y)- \Psi^{i}_{t, \boldsymbol v}(\bar y) = y -\bar y + t \big( \varphi^i(y)-\varphi^i(\bar y) \big) \boldsymbol v$
and defining 
\[ \varepsilon^i(y,\bar y) := \frac{\varphi^i(y)-\varphi^i(\bar y)}{|y-\bar y|},\]
as in the proof of \cite[Lemma 2.1]{CSV} we have, for $|t|$ small,
\begin{equation}\label{blabla1}
\big| K\bigl(z\pm t\varepsilon|z| \boldsymbol v \bigr) - K(z)  \mp t\partial_{\boldsymbol v} K(z) \varepsilon |z|  \big| \le C t^2\varepsilon^2 K(z)
\end{equation}
and
\begin{equation}
\label{aboveest}
\begin{split}
|\varepsilon^i(y,\bar y)| +& \frac{|J_t^i(y)-1|}t +  \frac{|J_t^i(\bar y)-1|}{t} \le 
\\
&\le
\left\{
\begin{array}{cc}
C & \text{if } i=0,
\vspace{3mm}
\\
\displaystyle  \frac{C}{\log R\, \max\bigl\{\sqrt R,\min(|y|,|\bar y|)\bigr\}} &\text{if } i=1,
\vspace{3mm}
\\
\displaystyle  \frac{C}{\log \log R \, \log \rho \,\rho} \qquad\mbox {for } \rho\ge R_*,\, |y|\ge\rho,\, |\bar y|\ge \rho  \quad &\text{if } i=2.\\
\end{array}
\right.
\end{split}
\end{equation}
Then, using \eqref{changeofvars}, \eqref{jac}, \eqref{blabla1}, and \eqref{aboveest}, and decomposing $A_R=A_{2^{2k}}=A_{2^k}\cup\left(\cup_{j=1}^k A_{2^{k+j}}\setminus A_{2^{k+j-1}}\right)$ when $i=1,$ an easy computation yields 

\[
\begin{split}
\mathcal E^{\rm Sob}(\mathcal P^i_{t,\boldsymbol v}v;  B_{R^i}) &+ \mathcal E^{\rm Sob}(\mathcal P^i_{-t,\boldsymbol v} v;B_{R^i}) - 2\mathcal E^{\rm Sob}(v; B_{R^i})\le
 \\
& \le 
\begin{cases}
\displaystyle
Ct^2\iint_{A_1}   |v(y)-v(\bar y)|^2 K(y-\bar y)\, dy\,d\bar y   & \text{if } i=0,
\vspace{3mm}
\\
\displaystyle
C\frac{t^2}{k^2}\sum_{j=1}^{k} \iint _{A_{2^{k+j}}} \frac{1}{ 2^{2(k+j)}}   |v(y)-v(\bar y)|^2 K(y-\bar y)\, dy\,d\bar y \quad &\text{if }i=1
\end{cases}
\end{split}
\]
(see the proof of \cite[Lemma 2.1]{CSV} for more details). Therefore \eqref{1} and \eqref{2} follow.

The proof of \eqref{3} needs a more careful estimate. For $\rho>0$, let us denote 
\[
e(\rho) :=\mathcal  E^{\rm Sob}(v; B_{\rho}) = \iint_{A_\rho}  |v(y)-v(\bar y)|^2 K(y-\bar y) \,dy\, d\bar y
\]
Note that 
\[
e'(\rho) = \lim_{h\downarrow 0}  \frac{1}{h}  \iint_{A_{\rho+h} \setminus A_\rho}  |v(y)-v(\bar y)|^2 K(y-\bar y)\, dy\, d\bar y.
\]
Observing that in the complement  of $A_\rho$ we have $|y|\ge\rho$ and  $|\bar y|\ge \rho$ and using \eqref{changeofvars}, \eqref{jac}, \eqref{blabla1}, and \eqref{aboveest} we obtain 
\[
\begin{split}
\mathcal E^{\rm Sob}(\mathcal P^2_{t,\boldsymbol v}v;  B_{R}) + \mathcal E^{\rm Sob}(\mathcal P^2_{-t,\boldsymbol v} v;B_{R}) &- 2\mathcal E^{\rm Sob}(v; B_{R})\le
\\
&\le 
\frac{Ct^2}
{(\log \log R)^2}  \left( \frac{e(R_*)}{ (\log R_*\, R_*)^2 } + \int_{R_*}^R \frac{e'(\rho)}{(\log \rho \, \rho)^2} d\rho\right)
\\\
&\le \frac{ Ct^2}{(\log \log R)^2}  \left( S+ C \int_{R_*}^R \frac{e(\rho)}{(\log \rho)^2 \, \rho^3} d\rho\right)
\\
&\le \frac{ Ct^2}{(\log \log R)^2}  \left( S+  \int_{R_*}^R \frac{S}{\log \rho \, \rho} d\rho\right)
\\
&\le  \frac{ Ct^2}{(\log \log R)} S, 
\end{split}
\]
where
\[
S:=  \sup_{\rho \ge 2} \frac{e(\rho)} {\rho^2\log \rho} = \sup_{\rho \ge 2} \frac{\mathcal E^{\rm Sob}(v;B_\rho)} {\rho^2\log \rho} ,
\]
so \eqref{3} follows.
\end{proof}

%\begin{proof}
%See the proofs of Lemma 2.1 in \cite{CSV} and Lemma 2.3 from \cite{CCS}.
%\end{proof}
%
%

The following is a basic BV estimate in $B_{1/2}$ for stable solutions in a ball.
\begin{lemma}\label{quadratic1}
Let $i\in \{0,1,2\}$, $R^i$ as in \eqref{eq:Ri}, and let $u\in C^{1,\alpha}(\overline{B_{R^i}})$ be a stable solution to $(-\Delta)^{1/2} u+f(u)=0$ in $B_{R^i}$
with $|u|\le 1$. Assume there exists $\eta>0$ such that, for $|t|$ small enough, we have
\begin{equation}\label{near2}
 \mathcal E(\mathcal P^i_{t,\boldsymbol v}  u; B_{R^i})+ \mathcal E(\mathcal P^i_{-t,\boldsymbol v} u; B_{R^i}) -  2\mathcal E(u; B_{R^i}) \le \eta t^2\qquad \forall\,\boldsymbol v\in \mathbb S^{d-1}.
\end{equation}
Then 
\begin{equation}\label{blabla}
\biggl(\int_{B_{1/2}}
\bigl(\partial_{\boldsymbol v} u(x)\bigr)_+\,dx\biggr)\biggl(\int_{B_{1/2}}
\bigl(\partial_{\boldsymbol v} u(y)\bigr)_-\,dy \biggr) \le 2\eta
\end{equation}
 and
\begin{equation}\label{eq:BV}
 \int_{B_{1/2}} |\nabla u| \le C(1+\sqrt \eta),  \end{equation}
for some dimensional constant $C$.
\end{lemma}

\begin{proof}
The proof is similar to the ones of  \cite[Lemmas 2.4 and 2.5]{CSV} or \cite[Lemma 2.5 and 2.6]{CCS}.
The key point is to note that, since $u$ is stable,
$$
\mathcal E(\mathcal P^i_{t,\boldsymbol v}  u; B_{R^i}) - \mathcal E(u; B_{R^i}) \geq -o(t^2),
$$
hence
\eqref{near2}
implies 
 \[
 \mathcal E(\mathcal P^i_{t,\boldsymbol v}  u; B_{R^i}) - \mathcal E(u; B_{R^i}) \le 2\eta  t^2
\]
for $|t|$ small enough.
 
On the other hand,   still by stability, the two functions 
\[
\overline u := \max\{u, P_{t,\boldsymbol v} ^0u\}\quad \mbox{and}\quad  \underline u := \min\{u, P_{t,\boldsymbol v} ^0u\}
\] 
satisfy 
\[
 \mathcal E(\overline u; B_{R^i}) - \mathcal E(u; B_{R^i})  \ge -o(t^2),\qquad 
  \mathcal E(\underline u; B_{R^i}) - \mathcal E(u; B_{R^i})  \ge -o(t^2).
\]
Hence, combining these inequalities with the identity 
\[\begin{split}
 \mathcal E(\overline u; B_{R^i})  + \mathcal E(\underline u; B_{R^i})  &+ 2\iint_{B_{R^i}\times B_{R^i}} \frac{ (\mathcal P^i_{t,\boldsymbol v}  u-u)_+(x)(\mathcal P^i_{t,\boldsymbol v}  u-u)_-(y)}{|x-y|^{d+1}}\,dx\,dy  \le 
 \\
  & \hspace{20mm}\le  \mathcal E(\mathcal P^i_{t,\boldsymbol v}  u; B_{R^i}) + \mathcal E(u; B_{R^i}) 
\end{split}
\]
we obtain
\[
2\iint_{B_{R^i}\times B_{R^i}  }\frac{ (\mathcal P^i_{t,\boldsymbol v}  u-u)_+(x)(\mathcal P^i_{t,\boldsymbol v}  u-u)_-(y)}{|x-y|^{d+1}}\,dx\,dy   \le 4\eta t^2
\]
Noticing that $\mathcal P^i_{t,\boldsymbol v} u(x) = u(x-t\boldsymbol v)$ for $x\in B_{1/2}$ and that $|x-y|^{-d-1} \ge 1$ for $x,y\in B_{1/2}$ we obtain the bound
$$
\iint_{B_{1/2}\times B_{1/2} } \left(\frac{u(x-t\boldsymbol v)-u(x)}{t}\right)_+\left(\frac{u(y-t\boldsymbol v)-u(y)}{t}\right)_- \,dx\,dy \leq 2\eta
$$
for all $|t|$ small enough, so \eqref{blabla} follows by letting $t \to 0$.

In other words, if we define 
$$
A_{\boldsymbol v}^\pm:=\int_{B_{1/2}}
\bigl(\partial_{\boldsymbol v} u(x)\bigr)_\pm\,dx,
$$
we have proved that $\min\{A_{\boldsymbol v}^+,A_{\boldsymbol v}^-\}^2\leq A_{\boldsymbol v}^+A_{\boldsymbol v}^- \leq 2\eta$.
In addition, since $|u|\leq 1$, by the divergence theorem
$$
|A_{\boldsymbol v}^+-A_{\boldsymbol v}^-|
=\biggl|\int_{B_{1/2}}\partial_{\boldsymbol v}u(x)\,dx\biggr| \leq \int_{\partial B_{1/2}}|\boldsymbol v\cdot \nu_{\partial B_{1/2}}|\leq C.
$$
Combining these bounds, this proves that
$$
\int_{B_{1/2}}|\partial_{\boldsymbol v}u(x)|\,dx=A_{\boldsymbol v}^++A_{\boldsymbol v}^-= |A_{\boldsymbol v}^+-A_{\boldsymbol v}^-|+2\min\{A_{\boldsymbol v}^+,A_{\boldsymbol v}^-\} \leq C+2\sqrt{2\eta},
$$
from which \eqref{eq:BV} follows immediately.
\end{proof}

We now recall the following general lemma due to Simon \cite{Simon} (see also  \cite[Lemma 3.1]{CSV}):
\begin{lemma}\label{lem_abstract}
Let $\beta\in \R$ and $C_0>0$. Let $\mathcal S: \mathcal B \rightarrow [0,+\infty]$ be a nonnegative function defined on the class $\mathcal B$  of open balls $B\subset \R^n$ and satisfying the following subadditivity property:
\[  B \subset \bigcup_{j=1}^N B_j \quad \Longrightarrow\quad \mathcal S(B)\le \sum_{j=1}^N \mathcal S(B_j). \]
Also, assume that $\mathcal S(B_1)< \infty.$
Then there exists $\delta = \delta(n,\beta)>0$ such that if
\begin{equation*}\label{hp-lem}
 \rho^\beta \mathcal S\bigl(B_{\rho/4(z)}\bigr) \le \delta \rho^\beta \mathcal S\bigl(B_\rho(z)\bigr)+ C_0\quad \mbox{whenever }B_\rho(z)\subset B_1
 \end{equation*}
then
\[ \mathcal S(B_{1/2}) \le CC_0,\]
where $C$ depends only on $d$ and $\beta$.
\end{lemma}

Finally, we state an optimal bound on the $H^{1/2}$ norm of the mollification of a bounded function with the standard heat kernel, in terms of the $BV$ norm and the parameter of mollification (see \cite[Lemma 2.1]{FJ} for a proof): 
%It gives an optimal bound on the $H^{1/2}(\R^d)$ norm of the mollification with the standard head kernel of a bounded BV function in terms of its $BV(\R^n)$ norm and the parameter of mollification.

\begin{lemma} \label{lemfig}
Let $H_{d,t}(x):= (4\pi t)^{-d/2} e^{-|x|^2/4t}$ denote the heat kernel in $\R^d$. 
Given $u\in BV(\R^d)$ with $|u|\le 1$, set $u_\varepsilon := u\ast H_{d,\varepsilon^2}$. 
Then, for $\varepsilon \in(0,1/2)$, we have
\[
 [ u_\varepsilon ]_{H^{1/2}(\R^d)}^2 := \iint_{\R^d\times\R^d} \frac{|u_\varepsilon(x)-u_\varepsilon(y)|^2}{|x-y|^{d+1}}\,dx\,dy \le C \log{\frac 1 \varepsilon}\, \|u\|_{BV(\R^d)},
\]
where $C$ is a dimensional constant.
\end{lemma}

\section{Proof of Proposition \ref{prop}}
\label{sect:prop}

As a preliminary result we need the following (sharp) interpolation estimate.
\begin{lemma}\label{newlemma}
Let $u:\R^d \to \R$ be a bounded function, with $|u|\leq 1$.
Assume that $u$ is Lipschitz in $B_2$, with $\|\nabla u\|_{L^\infty(B_2)}\le  L_o$ for some $L_o
\geq 2$.
Then
\[
\mathcal E^{\rm Sob} (u; B_1) \le C \log L_o\left( 1+ \int_{B_2} |\nabla u|\,dx\right)
\]
where $C$ depends only on $d$.
\end{lemma}
\begin{proof}
Let $\eta\in C^\infty_c(B_2)$, $0 \leq \eta \leq 1$, be a radial cutoff function
such that $\eta=1$ in $B_{3/2}$ and $\|\nabla \eta\|_{L^\infty(\R^d)}\leq 3$, and  set $\tilde u := \eta u$. 
Observe that, since $|u|\leq 1$, $0 \leq \eta \leq 1$, $\|\nabla u\|_{L^\infty(B_2)}\le  L_o$, and $\eta$ is supported inside $B_2$,
we have (recall that $L_0 \geq 2$)
\begin{equation}\label{Lip}
\|\nabla \tilde u\|_{L^\infty(\R^d)}\leq  \|\nabla u\|_{L^\infty(B_2)}
+\|\nabla \eta\|_{L^\infty(\R^d)} \leq L_0+3 \leq 3L_0.
\end{equation}
Now,  since $|u|\le 1$, we have
\begin{equation}\label{difwithtilde}
\mathcal E^{\rm Sob} (u; B_1) \le  \mathcal E^{\rm Sob} (\tilde u; B_1) + C
\end{equation}
where $C$ depends only on $d$.
On the other hand, it follows by Lemma \ref{lemfig} that
\begin{equation}\label{molli1}
 \| \tilde u_\varepsilon \|_{H^{1/2}(\R^d)}^2 \le C \log{\frac 1 \varepsilon}\, \|\tilde u\|_{BV(\R^d)}.
\end{equation}
We also observe that, because of \eqref{Lip},\footnote{The first inequality in \eqref{molli2} can be proven using Fourier transform,
noticing that
\[\widehat{(\tilde u_\epsilon-u)}(\xi) = (e^{-\varepsilon^2|\xi|^2}-1)\hat u(\xi),\]
and that $\frac{|e^{-\varepsilon^2|\xi|^2}-1|^2}{\varepsilon |\xi|}$ is universally bounded. Indeed,
\begin{multline*}
[\tilde u_\varepsilon -\tilde u]_{H^{1/2}(\R^d)}^2=\int |\xi|\,\bigl|\widehat{(\tilde u_\epsilon-u)}(\xi)\bigr|^2\,d\xi= \int |\xi|\,\bigl|e^{-\varepsilon^2|\xi|^2}-1\bigr|^2|\hat u(\xi)|^2\,d\xi\\
\leq C\varepsilon\int |\xi|^2|\hat u(\xi)|^2\,d\xi=C \varepsilon \|\nabla \tilde u\|_{L^2(\R^d)}^2.
\end{multline*}
}
\begin{equation}\label{molli2}
[ \tilde u_\varepsilon -\tilde u]_{H^{1/2}(\R^d)}^2  \le C \varepsilon \|\nabla \tilde u\|_{L^2(\R^d)}^2    \le C \varepsilon L_{o}\int_{\R^d} |\nabla \tilde u|\,dx. 
\end{equation}
Therefore, choosing  $\varepsilon=(L_o)^{-1}$ in \eqref{molli1} and \eqref{molli2}, and using a triangle inequality, we get (recall that $L_0\geq 2$)
\[
[\tilde u ]_{H^{1/2}(\R^d)}^2 \le 2 [ \tilde u_\varepsilon ]_{H^{1/2}(\R^d)}^2 +2 [ \tilde u_\varepsilon -\tilde u]_{H^{1/2}(\R^d)}^2
 \le C \log{L_o} \int_{\R^d} |\nabla \tilde u|\,dx .
\]
Finally, we note that 
\[
\mathcal E^{\rm Sob} (\tilde u; B_1) \le \mathcal E^{\rm Sob} (\tilde u; B_2) =  [\tilde u]_{H^{1/2}(\R^d)}^2
\]
and that (cp. \eqref{Lip})
\[
\int_{\R^d} |\nabla \tilde u|\,dx  \le C + \int_{B_2} |\nabla u|\,dx.
\]
Hence, recalling \eqref{difwithtilde}, we obtain
\[
\mathcal E^{\rm Sob} (u; B_1)  \le  C+ C\log{L_o} \left( 1+ \int_{B_2} |\nabla  u|\,dx\right),
\]
 and the lemma follows. %---recall that by assumption $L_o\ge 2$.
\end{proof}

We can now  prove Proposition \ref{prop}.
\begin{proof}[Proof of Proposition \ref{prop}]
This proof is  similar to the proof of Theorem 2.1 in \cite{CCS} (see also the proof Theorem 1.7 in \cite{CSV}). Here we need to use, as a new ingredient,  the estimate from Lemma \ref{newlemma}.
Throughout the proof, $C$ denotes a generic dimensional constant.\\

\noindent
- {\em Step 1}.
Let $v$ be a solution a stable solution  $(-\Delta)^{1/2} v + g(v)=0$ in $B_3$ satisfying $|v|\le 1$ in all of $\R^d$.

First,  using \eqref{1} in Lemma \ref{quadratic} and  then Lemma \ref{quadratic1} with $i=0$ and $R^i=1$, we obtain 
\begin{equation}\label{111}
\int_{B_{1/2}} |\nabla v| \,dx\le C\left(1 +\sqrt{ \mathcal E^{\rm Sob}(v; B_1)} \right). 
\end{equation}
Note that this estimate is valid for every stable solution $v$, independently of the nonlinearity $g$.

On the other hand, note that if $\|g\|_{C^{0,\alpha}([-1,1])}\le M_o$ for some $M_o\ge 2$, then by the interior regularity estimates for $(-\Delta)^{1/2}$ we have
\begin{equation}\label{eq:Lip v}
\|\nabla v\|_{L^\infty(B_2)} \le L_o := C M_o.
\end{equation}
Therefore, combining \eqref{111} with Lemma \ref{newlemma}, we obtain 
\begin{equation}\label{111bis}
\begin{split}
\int_{B_{1/2}}  |\nabla v|\,dx &\le C\left(1 + \sqrt{ C\log M_0 \left(1+ \int_{B_2} |\nabla v| \,dx\right)}\right)
\\
& \le \frac{C \log M_0 }{\delta} + \delta \int_{B_{2}}  |\nabla v|\,dx\qquad \forall\,\delta\in (0,1), 
\end{split}
\end{equation}
where we used the inequality $2\sqrt {ab}\le \delta a +b/\delta$ for $a,b\in \R_+$.\\

\noindent
- {\em Step 2}. For $v$ as in Step 1 and $B_{\rho}(z)\subset B_1$  we note that the function 
\[\tilde v (x) := v(z+2\rho x)\]
satisfies $(-\Delta)^{1/2} \tilde v + \tilde g(\tilde v) =0$ with $\tilde g(s)  := 2\rho g(s)$.
In particular  $\|\tilde g\|_{C^{0,\alpha}([-1,1])}\le 2\rho M_o \le 2M_o$,
so estimate \eqref{111bis} applied to $\tilde v$ yields
\[
\begin{split}
\int_{B_{1/2}}  |\nabla \tilde v|\,dx \le \frac{C \log (2M_0) }{\delta} + \delta \int_{B_{2}}  |\nabla \tilde v|\,dx,
\end{split}
\]
or equivalently
 \begin{equation}\label{111bis2}
\begin{split}
\rho^{1-d} \int_{B_{\rho/4}(z)}  |\nabla  v|\,dx \le \frac{C \log (2M_0) }{\delta} + \delta \rho^{1-d}  \int_{B_{\rho}(z)}  |\nabla v|\,dx. 
\end{split}
\end{equation}
Hence taking $\delta$ small enough and using Lemma \ref{lem_abstract} with $\mathcal S(B):= \int_B |\nabla v|\,dx$ and $\beta:=1-d$, we obtain
\begin{equation}\label{step2}
\int_{B_{1}} |\nabla  v|\,dx \le C \log M_0 ,
\end{equation}
where $C$ depends only on $d$.
Also, it follows by Lemma \ref{newlemma}
and \eqref{eq:Lip v} that
\begin{equation}\label{step2-2}
\mathcal E^{\rm Sob}(v,B_{1/2})
\leq C \log^2 M_0.
\end{equation}
\\

\noindent
- {\em Step 3}. If $u$ is a stable solution $u$ of $(-\Delta)^{1/2} u = f(u)$ in $B_{R}$, given 
$x_o\in B_{R/2}$ we consider the function
$v(x):= u\big(x_o+\frac R 6x\big)$.
Note that this function satisfies
\eqref{step2} and \eqref{step2-2} with $M_0$ replaced by $M_0R$, hence the desired estimates follow easily by scaling and a covering argument.
\end{proof}

\section{Proof of Theorem \ref{thm}}\label{sect:thm}
We are given $u$ a stable solution of $(-\Delta)^{1/2} u + f(u)=0$ in $\R^3$
with $|u|\le 1$   and $f\in C^{0,\alpha}$,
and we want to show that $u$ is 1D. 
We split the proof in three steps.\\

\noindent
- {\em Step 1.}
By Proposition \ref{prop} we have
\begin{equation} \label{growth}
\mathcal E^{\rm Sob}(u; B_R) \le CR^2 \log^2 R 
\end{equation}
for all $R\ge 2$, where $C$ depends only on $f$.
Take $k \geq 1$, $R=2^{2(k+1)}$, and $v(x):= u(Rx)$. Note that, by elliptic regularity, $\|\nabla u\|_{L^\infty(\R^3)} \leq C$ for some constant $C$ depending only on $f$,
thus
$\|\nabla v\|_{L^\infty(\R^3)}\le CR$.
Also, $v$ is still a stable solution of a semilinear equation in all of $\R^3$.
Hence, using \eqref{2} in Lemma \ref{quadratic}  and then Lemma \ref{quadratic1}  with  $i=1$ and $R^i =R$, we obtain 
\begin{equation}\label{222}
\int_{B_{1/2}} |\nabla v| \,dx\le C\left(1 +\sqrt{ \frac1{k} \sum_{j=1}^{k} \frac{\mathcal E^{\rm Sob}(v; B_{2^{k+j}}) }{ k 2^{2(k+j)}} }\right). 
\end{equation}
On the other hand, using Lemma \ref{newlemma} and the bound $\|\nabla v\|_{L^\infty(\R^3)}\le CR$, we have
\begin{equation}\label{333}
\frac{\mathcal E^{\rm Sob}( v; B_{1/4})   }{  \log R } \le C
\left(1+\int_{B_{1/2}} |\nabla v| \,dx\right).
\end{equation}
Thus, recalling that  $R=2^{2(k+1)}= 4\cdot 2^{2k}$, $v(x)= u(Rx)$, rewriting \eqref{222}
and \eqref{333} in terms of $u$ we get (here we use that $d=3$) 
\begin{equation}\label{444}
\begin{split}
\frac{\mathcal E^{\rm Sob}( u; B_{2^{2k}})   }{  k2^{4k} }&\le C\left(1 +\sqrt{  \frac{1}{k}\sum_{j=1}^{k} \frac{\mathcal E^{\rm Sob}(u; B_{2^{3k+j+2}}) }{ k2^{2(k+j) }2^{4(k+1)}} }\right) 
\\&\le C\left(1 +\sqrt{  \frac{1}{2k}\sum_{\ell=1}^{2k} \frac{\mathcal E^{\rm Sob}(u; B_{2^{2(k+\ell)}}) }{ (k+\ell)2^{4(k+\ell)}} }\right) \qquad \forall\,k \geq 1.
\end{split}
\end{equation}

\noindent
- {\em Step 2.} Given $j \geq 1$ set
\[\mathcal A(j): =  \frac{\mathcal E^{\rm Sob}(u; B_{2^{2j}}) }{ j2^{4j}},\] 
so that
\eqref{444} can be rewritten as
\begin{equation}\label{555}
\mathcal A(k) \le C_f\left(1 +\sqrt{  \frac{1}{2k}\sum_{\ell=1}^{2k}  \mathcal A(k+\ell)}\right). 
\end{equation}
where $C_f$ depends only on $f$.
We claim that
\begin{equation}\label{goalstep2}
\mathcal A (k) \le M \quad \mbox{for all }k\ge 1
\end{equation}
for some  constant $M$ depending only on $f$.

Indeed, assume by contradiction that $\mathcal A(k_0) \ge M$ for some large constant $M$ to be chosen later. Rewriting \eqref{555} as
$$
2c_{f} \mathcal A(k_0) -1 \le  \sqrt{  \frac{1}{2k_0}\sum_{\ell=1}^{2k_0}  \mathcal A(k_0+\ell)},\qquad c_f :=\frac{1}{2C_f},
$$
then, provided $M\geq \frac{1}{c_f}$,
if $\mathcal A(k_0) \ge M$ we find
\begin{equation*}\label{6666}
c_f M \le   \sqrt{  \frac{1}{2k_0}\sum_{\ell=1}^{2k_0}  \mathcal A(k_0+\ell)}. 
\end{equation*}
This implies that there exists $k_1\in\{k_0+1,\ldots, 3k_0\}$  such that 
\[
c_f M \le  \sqrt{ \mathcal A(k_1)},
\]
that is
\[
(c_f M)^2 \le  \mathcal A(k_1). 
\]
Hence, choosing $M$ large enough so that $\tilde M := (c_fM)^2 \ge \frac{1}{c_f}$, we can repeat exactly the same argument as above with $M$ replaced by $\tilde M$ and $k_0$ replaced by $k_1$ in order to find  $k_2 \in \{k_1+1,3k_1\}$  such that
\[
(c_f \tilde M)^2  = c_f^6 M^4\le  \mathcal A(k_2). 
\]
Iterating further we find $k_1< k_2,<k_3< \ldots<k_m<\ldots$ such that $k_{m+1}\le 3k_m$ and 
\[
c_f^{2^{m+1}-2}M^{2^m} \le  \mathcal A(k_m). 
\]
Now, ensuring that $M$ is large enough so that $\theta :=  c_f^2 M >1$, we obtain 
\begin{equation}\label{expg}
c_f^{-2}\theta^{2^m} \le  \mathcal A(k_m). 
\end{equation}
On the other hand, recalling \eqref{growth} and using that $k_m\le 3^{m} k_0$, we have 
\begin{equation}
\label{ling}
\mathcal A(k_m) \le C \log (3^m k_0) \le C(m+ \log k_0).
\end{equation}
The linear bound from \eqref{ling} clearly contradicts the exponential growth in \eqref{expg} for $m$ large enough. Hence, this provides the desired contradiction and proves
\eqref{goalstep2}\\

\noindent
- {\em Step 3.} Rephrasing \eqref{goalstep2}, we proved that
\begin{equation}\label{egest}
\mathcal E^{\rm Sob}( u; B_{R})   \le CR^2\log R
\end{equation}
for all $R\ge 2$, where $C$ depends only on $f$. In other words, we have 
obtained an optimal energy estimate in large balls $B_R$ 
(note that 1D profiles saturates \eqref{egest}).
Having improved the energy bound of Proposition \ref{prop} from $R^2\log^2R$ to  \eqref{egest},
we now conclude that $u$ is a 1D profile as follows.

Given $\boldsymbol v\in \mathbb S^{d-1}$ and using the perturbation $\mathcal P^2_{t,\boldsymbol v}$ as in \eqref{perturb}-\eqref{perturb2}, it follows by 
\eqref{3}, \eqref{near2}, and \eqref{blabla} that
$$
\iint_{B_{1/2}\times B_{1/2} }
\bigl(\partial_{\boldsymbol v} u(x)\bigr)_+
\bigl(\partial_{\boldsymbol v} u(y)\bigr)_-\,dx\,dy  \le  \frac{C}{\log\log R}.
$$
Hence, taking the limit 
as $R\to \infty$ we find that 
$$
\iint_{B_{1/2}\times B_{1/2} }
\bigl(\partial_{\boldsymbol v} u(x)\bigr)_+
\bigl(\partial_{\boldsymbol v} u(y)\bigr)_-\,dx\,dy =0,
$$
thus
$$
\text{either}\qquad\partial_{\boldsymbol v} u \ge 0 \quad \mbox{in } B_{1/2} \qquad \mbox{or} \qquad \partial_{\boldsymbol v} u \le 0  \quad\mbox{in } B_{1/2}\qquad \forall\,\boldsymbol v\in \mathbb S^{d-1}.
$$
Since this argument can be repeated changing the center of the ball $B_{1/2}$ with any other point, by a continuity argument we obtain that
$$
\text{either}\qquad\partial_{\boldsymbol v} u \ge 0 \quad \mbox{in } \R^3 \qquad \mbox{or} \qquad \partial_{\boldsymbol v} u \le 0  \quad\mbox{in } \R^3\qquad \forall\,\boldsymbol v\in \mathbb S^{d-1}.
$$
Thanks to this fact, we easily conclude that $u$ is a 1D monotone function, as desired. 
\qed

%%%%

\end{document}